\documentclass[12pt]{amsart}
                                                                                                                                                                                                                                                                                                                                                                                                                                                            \usepackage{amsfonts,amsmath,amssymb}
\usepackage{latexsym,amscd,graphicx}
\usepackage[cmtip,all]{xy}

	\newtheorem{thm}{Theorem}[section]
\theoremstyle{definition}
\newtheorem{defn}[thm]{Definition}

\newtheorem{prop}[thm]{Proposition}
\newtheorem{lem}[thm]{Lemma}
\newtheorem{rem}[thm]{Remark}
\newtheorem{cor}[thm]{Corollary}

\newcommand{\s}{\mathfrak{sl} _2(\mathbb{C})}

\newcommand{\zzz}{ \mathbb{Z} }
\newcommand{\cc}{ \mathbb{C} }

\newcommand{\aaa}{ \mathbb{A} }

\newcommand{\ii}{ I_{\lambda}}
\newcommand{\ui}{ U(I_{\lambda}) }
\newcommand{\h}{ {h} }

\newcommand{\B}{ {B} }

\newcommand{\mc}{ M_{\lambda}^{C} }
\newcommand{\nc}{ N_{\lambda}^{C} }
\newcommand{\wc}{ P }
\newcommand{\wcc}{ Q }

\newcommand{\wh}[1]{\widehat{#1}}

\newcommand{\ld}{\ldots}

\newcommand{\beq}{\begin{equation}}
\newcommand{\eeq}{\end{equation}}
\newcommand{\beas}{\begin{eqnarray*}}
\newcommand{\eeas}{\end{eqnarray*}}
\newcommand{\id}{\mathrm{id}}

\newcommand{\C}{\mathbb{C}}

\newcommand{\z}{\mathbb{Z}_2}

\newcommand{\ZZ}{\mathbb{Z}}

\newcommand{\CC}{\mathbb{C}}

\newcommand{\FF}{\mathbb{F}}

\DeclareMathOperator{\Supp}{\mathrm{Supp}}

\newcommand{\la}{\langle}
\newcommand{\ra}{\rangle}

\newcommand{\Sl}{\mathfrak{sl}}

\usepackage{array}
\newcolumntype{L}{>{$}l<{$}}

\begin{document}

\title{\sc{Graded torsion-free $ \s $-modules of rank 2}}

\author[Bahturin]{Yuri Bahturin}
\address{Department of Mathematics and Statistics, Memorial
University of Newfoundland, St. John's, NL, A1C5S7, Canada}
\email{bahturin@mun.ca}

\author[Shihadeh]{Abdallah Shihadeh}
\address{Department of Mathematics and Statistics, Memorial
University of Newfoundland, St. John's, NL, A1C5S7, Canada}
\email{aaks47@mun.ca}

\thanks{{\em Keywords:} graded Lie algebras, enveloping algebras, simple modules, graded modules}
\thanks{{\em 2010 Mathematics Subject Classification:} Primary 17B70, Secondary 17B10, 17B20, 17B35.}
\thanks{The first author acknowledges support by NSERC grant \# 227060-14. The second author acknowledges support by NSERC grant \# 227060-14 and NSERC grant \# -18.}

\begin{abstract} In this paper we explore the possibility of endowing simple infinite-dimensional $\s$-modules by the structure of the graded module. The gradings on finite-dimensional simple module over simple Lie algebras has been studied in \cite{EK_Israel} and \cite{EK_loop}.

\end{abstract}

\maketitle

\section{Introduction}\label{sI}

Let $G$ be a non-empty set. A  vector space $V$ over a field $\FF$ is called $G$-\textit{graded} if it can be written as the  direct sum of subspaces indexed by $G$, as follows: 
\begin{equation}\label{eI1}
V=\bigoplus_{g\in G}V_g.
\end{equation}

We will sometimes use Greek letters to refer to the gradings, for example, 
\[
\Gamma: V=\bigoplus_{g\in G}V_g.
\]

Generally, it is allowed that some of $V_g$ are zero subspaces. The subset $S\subset G$ consisting of those $g\in G$ for which $V_g\ne\{ 0\}$ is called the \textit{support} of the grading $\Gamma$. One writes $S=\Supp \Gamma$ or $S=\Supp V$. The subspaces $V_g$ are called the \textit{ homogeneous components} of $\Gamma$. Nonzero elements in $V_g$ are called \textit{homogeneous of degree} $g$. An $\FF$-subspace $U$ in $V$ is called \textit{graded} in $V$ (or in $\Gamma$) if $U=\bigoplus_{g\in G}U\cap V_g$. 

Now let $ \Gamma $ and $ \Gamma^{\prime}: V=\bigoplus_{g'\in G'}V^\prime_{g'} $ be two gradings on $ V $ with supports $ S $  and $ S' $, respectively. We say that $ \Gamma $ is a \emph{refinement} of $ \Gamma^{\prime} $ (or $ \Gamma^{\prime} $ is a \emph{coarsening} of $ \Gamma $), if for any $ s\in S $ there exists $ s'\in S' $ such that $ V_{s}\subseteq V^{\prime}_{s'} $. The refinement is \emph{proper} if this inclusion is strict for at least one $ s\in S $. A grading is called \emph{fine} if it does not have proper refinements.

An $\FF$-algebra $A$ (not necessarily associative) is said to be \emph{graded by a set $G$}, or \emph{$G$-graded} if $A$ is a $G$-graded vector space and for any $g,h\in G$ such that $A_gA_h\ne\{ 0\}$ there is $k\in G$ such that    
\begin{equation}\label{eg1}
A_g A_h\subset A_k.
\end{equation} 
In this paper, we will always assume that $G$ is an \emph{abelian group} and $k$ in Equation \eqref{eg1} is determined by the operation of $G$. Thus, if $G$ is written additively (as is commonly done in the papers on Lie theory), then Equation \eqref{eg1} becomes $A_g A_h\subset A_{g+h}$. If $G$ is written multiplicatively, then it becomes $A_g A_h\subset A_{gh}$.

	Given a grading $\Gamma: A=\bigoplus_{g\in G}A_g$ with support $S$, the \textit{universal (abelian) group} of $\Gamma$, denoted by $G^u$, is the abelian group given in terms of generators and defining relations as follows: $G^u=(S\:\vert\:R)$, where $gh=k$ is a relation in $R$ if $\{ 0\}\ne A_gA_h\subset A_k$. If $\Gamma$ is a group grading then the identity map $\id_S$ extends to a homomorphism of $G^u$ to $G$ so that $\Gamma$ can be viewed as a $G^u$-grading. Any coarsening $\Gamma':A=\bigoplus_{g'\in G'}A^\prime_{g'} $ of $\Gamma$ is a quotient-grading of this $G^u$-grading of $A$ in the sense that there exists a unique homomorphism $ \nu:G^u\to G' $ such that $  A_{g'}=\bigoplus_{g\in\nu^{-1}(g') }A_g $, for any $g'\in G'$. 
	
	For example, consider the standard basis of $L=\Sl_2(\C)$ 
	 \begin{equation*}\label{sl2}
	x=\begin{bmatrix}
	0 &1 \\
	0 &0
	\end{bmatrix},\quad h=
	\begin{bmatrix}
	1 &0 \\
	0 &-1
	\end{bmatrix},\quad y=
	\begin{bmatrix}
	0 &0 \\
	1 &0 
	\end{bmatrix}.
	\end{equation*}
   and give $L$ the following  grading by $G=\ZZ_3$:
\[
\Gamma:\; \;L_0=\la h\ra, \;L_1=\la x\ra, \;L_2=\la y\ra.
\]
The support is $G$ itself.

The following grading by $G'=\ZZ_2$ is a coarsening of $ \Gamma $ : 
\[
\Gamma':\; L_1=\la x,y\ra,\;L_0=\la h\ra.
\]
The universal group of $\Gamma$ is $\ZZ$, with
\[
\Gamma_0:\; L_{-1}=\la x\ra,\;L_0=\la h\ra,\;L_1=\la y\ra.
\]
Both $\Gamma$ and $\Gamma'$ are factor-gradings of $\Gamma_0$; $\Gamma'$ is a coarsening but not a factor-grading of $\Gamma$.

 A left module $M$ over a \emph{$G$-graded} associative algebra $A$ is called \textit{$G$-graded} if $M$ is a $G$-graded vector space and 
\[
A_g M_h\subset M_{g+h}\mbox{ for all }g, h\in G.
\]
(Later on, the multiplicative notation for the graded modules will be used, as well).

A left $G$-graded $A$-module $M$ is called \textit{graded-simple} if $M$ has no graded submodules different from $\{ 0\}$ and $M$. Graded modules and graded-simple modules over a graded Lie algebra $L$ are defined in the same way.

If a Lie algebra $L$ is $G$-graded then its universal enveloping algebra $U(G)$ is also $G$-graded. Every graded $L$-module is a left $U(L)$-graded module, and \textit{vice versa}.  The same is true for simple graded $L$-modules and graded-simple $L$-modules.

An $ \s $-module is called \textit{torsion-free} if it is torsion-free as an $\CC[h]$-module. Every simple module is either torsion -free or torsion, also called \textit{weight} module. Torsion-free $ \s $-modules has been studied in \cite{Bavula classif,irr..rep..block,arb rank,Mazorchuk,u(h)struct}. In 1992 Bavula constructed a family of simple torsion-free modules \cite{Bav weyl}.  The torsion-free $ \s $-modules of rank $ 1 $ have been classified in \cite{u(h)struct}. In \cite{arb rank}, the author produced torsion-free modules of arbitrary finite rank. 

The gradings of the torsion-free $ \s $-module of rank $ 1  $ have been dealt with in \cite{BKS}; the main result was that the torsion-free $ \s $-modules of rank $ 1 $ cannot be $ \zzz$- or $ \zzz_{2}^{2} $-graded. In \cite{BKS} one also considered the gradings of the family of modules constructed in \cite{Bav weyl}.   

In this paper, our main goal is to construct a simple torsion-free $ \zzz_{2}^{2} $-graded module. For this, we consider the following module (see the notations in Sections \ref{ssGGS} and \ref{ssul}): 
\begin{equation*}
\mc:=\ui/\ui C,
\end{equation*}
where $ \lambda\in \cc $. The main result in this paper that the graded module $ \mc $ is simple when $ \lambda $ is not an even integer. Otherwise, the module $ \mc $ contains a proper graded simple torsion-free module of rank $ 2 $.
  
\section{Group gradings of $\Sl_2(\C)$-modules}\label{sggsl2}  In this section we restrict our attention to the modules over the Lie algebra of the type $A_1$, which can be realized as $ \s $.

\subsection{Group gradings of $ \s $}\label{ssGGS}

All group gradings on $\Sl_2(\C)$ are well-known, see e.g \cite{gradbook}. We will use the following bases:

\begin{equation}\label{sl2g}
x=\begin{bmatrix}
0 &1 \\
0 &0
\end{bmatrix},\quad h=\begin{bmatrix}
1 &0 \\
0 &-1
\end{bmatrix},\quad y=\begin{bmatrix}
0 &0 \\
1 &0 
\end{bmatrix},
\end{equation}
as earlier, and
\begin{equation}\label{sl20}
A=\begin{bmatrix}
1 &0 \\
0 &-1
\end{bmatrix}=h,\quad B=\begin{bmatrix}
0 &1 \\
1 &0
\end{bmatrix},\quad C=\begin{bmatrix}
0 &1 \\
-1 &0 
\end{bmatrix}.
\end{equation}  
 
 Up to equivalence, there are precisely two fine gradings on $ \s $ (see \cite[Theorem 3.55]{gradbook}):
 \begin{itemize}
		\item \emph{Cartan grading} with the universal group $\mathbb{Z}$,
		\begin{eqnarray*}
			\Gamma^1_{\Sl_2} :\mathfrak{sl}_2(\mathbb{C})&=&L_{-1}\oplus L_0\oplus L_1\\ \mbox{ where }L_0&=&\left\langle h \right\rangle,\, L_1=\left\langle x \right\rangle,\, L_{-1}=\left\langle  y\right\rangle;
			\end{eqnarray*}
		
		\item \emph{Pauli grading} with the universal group $\mathbb{Z}^2_2$, 
	\begin{eqnarray*}
			\Gamma^2_{\Sl_2} :\mathfrak{sl}_2(\mathbb{C})&=&L_{(1,0)}\oplus L_{(0,1)}\oplus L_{(1,1)}\\\mbox{ where }L_{(1,0)}&=&\left\langle A \right\rangle,\,L_{(0,1)}=\left\langle B \right\rangle,\, L_{(1,1)}=\left\langle C\right\rangle .
			\end{eqnarray*}
		\end{itemize}
	 Hence, up to isomorphism, any $ G $-grading on $ \s $  is a coarsening of one of the two gradings: Cartan or Pauli. 
	 
	 Note that any grading $\Gamma$ of a Lie algebra $L$ uniquely extends to a grading $U(\Gamma)$ of its universal enveloping algebra $U(L)$. The grading $U(\Gamma)$ is a grading in the sense of associative algebras but also as $L$-modules where $U(L)$ is either a (left) regular $L$-module or an adjoint $L$-module. In the study of gradings on $\Sl_2(\C)$-modules one often considers a $\ZZ_2$-coarsening of $U(\Gamma^2_{\Sl_2})$, in which the component of the coarsening labeled by $0$ is the sum of components of the original grading labeled by $(0,0)$ and $(1,0)$ while  the component labeled by $1$ is the sum of components labeled by $(0,1)$ and $(1,1)$.

\subsection{Algebras $U(I_\lambda)$}\label{ssul}
 Let $ c\in U(\s)$ be the Casimir element for $ \s $. With respect to the basis $ \left\lbrace h,x,y \right\rbrace  $ of $ \s $, this element can be written as
 \begin{equation}
 c=(h+1)^{2}+4yx=h^{2}+1+2xy+2yx.
 \end{equation} 
 It is will-known that the center of $ U({\s}) $ is the polynomial ring $\cc[c]  $. Note that $ c $ is a homogeneous element of degree zero, with respect to the Cartan grading of $U(\s)$. One can write the Casimir element with respect to the basis $ \left\lbrace h,B,C \right\rbrace $  of $ \s $. 
 
 Namely, 
 \begin{eqnarray*}
  c&=& 2xy+2yx+h^{2}+1=2\left(\frac{B+C}{2}\right)\left(\frac{B-C}{2}\right)\\
  &+&2\left(\frac{B-C}{2}\right)\left(\frac{B+C}{2}\right)+h^{2}+I\\
  &=&\frac{1}{2}(B^{2}+CB-BC-C^{2})+\frac{1}{2}(B^{2}+BC-CB-C^{2}) +h^{2}+I, 
 \end{eqnarray*}
 and so 
\[
 c=B^{2}-C^{2}+h^{2}+I=A^2+B^2-C^2+I.
\]
 It follows that $ c $ is also homogeneous, of degree (0,0), with respect to the Pauli grading of $U(\s)$. Note that this  basis of the center of the universal enveloping algebra is a very  particular case of  the bases given in \cite{BM}.
 
 Let $ R $ be an associative algebra (or just an associative ring), and $ V $ be a left $ R $-module. The annihilator of $ V $, denoted by $ \mathrm{Ann}_{R}(V) $, is the set of all elements $ r $ in $ R $ such that, for all $ v $ in $ V $, $ r.v=0: $ 
 \begin{center}
 	$\mathrm{Ann}_{R}(V)=\left\lbrace r\in R \mid r.v=0 \text{ for all } v\in V \right\rbrace $.
 \end{center} 
 Given $ \lambda\in \cc$, let $ I_{\lambda} $ be the two-side ideal of $ U( \s)$, generated by the central element $ c-(\lambda+1)^{2} $.
 
 \begin{thm}\cite[Theorem 4.7]{Mazorchuk}\label{ann of simple module}
 	For any simple $ U({\s}) $-module $ M $, there exists $ \lambda\in  \cc$ such that $ I_{\lambda}\subset\mathrm{Ann}_{U({\s})}(M) $.
 \end{thm}
 
 Clearly, if $ R $ be a graded algebra and $ M $ be a graded $ R $-module, then $\mathrm{Ann}_{R}(M)$ is a graded ideal. 

 \begin{prop}[\cite{BKS}]\label{I is graded}
 	The ideal $ I_{\lambda} $ is both  $ \ZZ- $ and $\z ^{2}- $graded ideal.
 \end{prop} 
\
 Now for any $ \lambda\in \cc$, we write $\ui := U({\s})/{\ii} $. Using Proposition \ref{I is graded}, it follows that $\ui$ \  has natural $ \ZZ $- and $\z ^{2} $-gradings induced from $L$. It is well-known (see e.g \cite{Mazorchuk}) that the algebra $\ui$ is a free $\cc[h]  $-module with basis $ \mathcal{B}_{0}=\left\lbrace 1,x,y,x^{2},y^{2},\ldots\right\rbrace  $, and so it is a vector space over $ \cc $ with basis $ \mathcal{B}=\left\lbrace 1,h,h^{2},\ldots \right\rbrace .\mathcal{B}_{0}$.
 Note that the basis $\mathcal{B} $ is a basis of $\ui$\ consisting of homogeneous element with respect to the Cartan grading by $ \ZZ $. 
 
 A basis of $\ui$\ over $ \cc $ consisting of homogeneous element with respect to the Pauli grading by $\z^{2} $ can  be computed as follows.
Set $\wh{B}_0=\{1,B,C,BC,B^2,B^2C,B^3,B^3C,\ld\}$. Then easy calculations, using the induction by the degree in $B$ and the relation $C^2=h^2+B^2-\lambda^{2}-2\lambda$ show that the set $\wh{B}=\{1,h,h^2,\ld\}\cdot\wh{B}_0$ is a $\ZZ_2^2$-homogeneous basis of $\ui$.
 
 Now let $ p(t) =\frac{1}{4}((\lambda^{2}+2\lambda)-2t-t^{2})\in{\cc}[t] $. Then, inside $\ui$\, for any $q(t)\in\C[t]$, we have the following relations:
 \begin{center}
 	$ x^{k}q(h)=q(h-2k)x^{k} $\\
 	$ y^{j}q(h)=q(h+2j)x^{j} $ 
 \end{center}
 If $k\geq j  $ then
 \begin{center}
 	$ x^{k}y^{j}=p(h-2k)\cdots p(h-2(k-j+1))x^{k-j} $\\
 	$ y^{j}x^{k}=p(h+2(j-1))\cdots p(h)x^{k-j} $\\
 \end{center}
 If $ j\geq k $ then
 \begin{center}
 	$ x^{k}y^{j}=p(h-2k)\cdots p(h-2)y^{j-k} $\\
 	$ y^{j}x^{k}=p(h+2(j-1))\cdots p(h+2(j-k))y^{j-k} $.\\
 \end{center} 

 One more property that is important for us is the following.
  
\begin{thm}\cite[Theorem 4.26]{Mazorchuk} \label{finite lenght }
	For any non-zero left ideal $ I \subset  \ui$, the $\ui$-module $\ui /I $ has finite length.
	
\end{thm}
 \begin{prop}
 	Set $\wh{B}_0=\{1,B,C,BC,B^2,B^2C,B^3,B^3C,\ld\}$. \!\!Then $\wh{B}=\{1,h,h^2,\ld\}\cdot\wh{B}_0$ is a $\ZZ_2^2$-homogeneous basis of $ \ui $.	
 \end{prop} 
 \begin{proof}
 	Consider the natural ascending filtration $\{ U^{(n)}\;|\;n=0,1,2,\ldots\}$ of $ U(\s) $, where $ U^{(n)} $ is the linear span of monomials of the form $ e_{1}e_{2}\cdots e_{k} $ where $ e_{i}\in \s $ for all $ i $ and $ k\leq n $. Using PBW Theorem, if $ \left\lbrace  f_{1},f_{2},f_{3}\right\rbrace  $ be an ordered basis of $ \s $, then this filtration can just be  the linear span of the standard monomials of degree less than or equal $ n $, that is the monomials of the form $ f_{1}^{i}f_{2}^{j}f_{3}^{l} $ where the degree of the monomial is the sum of the powers.
 	Consider now the filtration $\{ I^{(n)}\;|\;n=0,1,2,\ldots\}$ of $ \ui $  induced by the natural filtration of $U( \s) $. Using the above basis $\mathcal{B}$, we can write that the basis of $I^{(n)}$ in the form of the set 
 \begin{equation}\label{In}	
 	\{ h^kx^\ell\;|\; 0\le k,\ell\le n,\:k+\ell\le n\}\cup\{ h^ky^\ell\;|\; 0\le k,\ell\le n\,\:k+\ell\le n\}.
 	\end{equation}
  Now using (\ref{sl20}), we can assume that the basis of $I^{(n)}$ is formed by some of the monomials  $\{ h^{k}B^{l}C^{m}\mid k+l+m\leq n\}$.
 	
Let us prove that actually $I^{(n)}$ is spanned by the same monomials such that $ m=0 \text{ or } 1$.  Let us use induction on $n$, with an obvious basis, and the relation  $ C^{2}=h^{2}+B^{2}-\mu$, where $ \mu=\lambda^{2}+2\lambda $, holding in $ \ui $. As a result,
it is sufficient to deal with the elements $ h^{k}B^{l}C^{2}$ with $k+\ell+2=n$. We write 
\[
 h^{k}B^{l}C^{2}=h^{k}B^{l}(h^{2}+B^{2}-\mu)=h^{k}(B^{l}h^{2}+B^{l+2}-\mu B^l).
 \]
Since \[h^{k}B^{l}h^{2}=h^{k+2}B^l+h^{k}[B^{l},h^{2}]\] where the second term belong to $ I^{(s)} $ for some $ s<n $, we can use the induction step and write $ h^{k}B^{l}C^{2} $  as a sum of the terms of the desired form. Now it follows that $I^{(n)}$ is spanned by the set of monomials
 \begin{equation}\label{Inn}
\left\langle h^{k}B^{l}C^{m}\mid k+l+m\leq n \text{ and } m=0,1\right\rangle.
\end{equation}
 	Using the basis (\ref{In}), we find that the dimension of  $ I^{(n)} $ equals $\sum_{i=0}^{n}( 2i+1) $. Clearly, the number of elements in (\ref{Inn}) is the same and hence  a basis of  $\ui$ is formed by the monomials $\{ h^kB^\ell\}\cup\{ h^kB^\ell C\}$ where $ k,\ell=0,1,2,\ldots$
 \end{proof}

\subsection{$ \s $-modules: torsion-free modules}\label{ssltf}
\begin{defn}
	Following \cite{Mazorchuk}, an $ \s $-module $M$ is called  \textit{torsion} if for any $ m\in M $ there exists non-zero $ p(
	t)\in {\cc}[t] $ such that $ p(h).m =0$. In other words, $h$ has an eigenvector in $M$. We call $ M $ \emph{torsion-free} if $ M\neq 0 $ and $ p(h).m\neq 0 $ for all $ 0\neq m\in M$ and all non-zero $p(t)\in {\cc}[t] $. In other words, $M$ is a free $\C[h]$-module. If the rank of this module is $n$, we say that $M$ is a torsion-free $\s$-module  of \textit{rank} $ n $. 
\end{defn}

\begin{thm}\cite[Theorem 6.3]{Mazorchuk} \label{simple eather weight or free}
	A simple $ \s $-module is either a weight or a torsion-free module. \hfill $ \Box $
\end{thm}
Theorem \ref{simple eather weight or free} means that if $ h  $ has at least one eigenvector on $ M $, then $ M $ is a weight module.

 As a consequence of Theorem \ref{ann of simple module}, it is sufficient to describe simple torsion-free $\ui$-modules instead of simple $ U({\s})$-modules  (see e.g \cite{Mazorchuk}).
 
The classification of simple torsion free $\ui$-module is given in the major paper \cite{irr..rep..block}. In \cite{Bavula classif} the author uses a certain localization $ \aaa $ of $\ui$ to obtain the following.

\begin{thm}\cite[Proposition 3]{Bavula classif}
	Let $ M  $ be a simple torsion-free $\ui$-module, then $ M\cong \ui/(\ui\cap {\aaa}\alpha) $, for some $ \alpha\in \ui$\ which is irreducible as an element of $ \aaa $. 
\end{thm}

\begin{thm}[\cite{BKS}]\label{any rank 1}
Torsion free $ \s $-modules of rank $ 1 $ cannot be $ \zzz $ or $\zzz ^{2}_{2} $-graded.
\end{thm}

\section{Simple graded $ \s $-modules of rank $ 2 $ }
In this section we will study simple torsion-free $ \s $-modules of rank $ 2 $. In particular, we will construct the first family of simple $ \zzz_{2}^{2} $-graded torsion-free $ \s $-modules of rank $ 2 $.

	Let $ \lambda\in \cc $, and consider the $ \ui$-module  $ \mc=\ui/\ui C$. For $ u,v \in U(\s)$, we say that $ u\equiv v $ and say that $u$ is equivalent to $v$,  if and only if $ u+\ui C=v+\ui C $. All elements of the form	$ {h}^{k}{B}^{l}{C}$ are equivalent to 0. Moreover  $ {B}^{2}\equiv\mu-{h}^{2} $ where $ \mu=\lambda^{2}+2\lambda $ . Hence \[{h}^{k}{B}^{2}\equiv{h}^{k}(\mu-{h}^{2})=\mu {h}^{k}-{h}^{k+2},\] which implies that any element of $ \mc $ can be written as a linear combination of elements of the form $ {h}^{k}{B}^{m} $ where $ m=0,1 $. This means that $ \mc $ is a torsion-free $ \s $-module of rank $ 2 $, with basis $ \left\lbrace {1},{B}\right\rbrace  $ as a $ \cc[h] $-module, that is, $ \mc $ can be identified as $ \mc=\cc[h]\oplus\cc[h]{B} $. Note that $ \mc  $ is graded module since $ \ui C $ is a graded left ideal.
	
		\begin{thm}[Main Theorem 1]\label{mc simple}
		Let $ \lambda\in \cc\setminus 2\zzz $, then $ \mc $ is a simple $ \s $-module .
	\end{thm}

	\begin{thm}[Main Theorem 2]\label{mc not simple}
		Let $ \lambda\in 2\zzz $, then $ \mc $ has a unique maximal (graded) submodule $ \nc $ such that $ \nc=P\oplus Q $, where $ P $ and $ Q $ are simple $ \s $-module of rank $ 1 $.
	\end{thm}
	
	Before we prove the above theorems, we will start with some necessary calculations and relations.
	
	Let $ f({h})\in \cc[h]\subset\mc $, then

	 \begin{eqnarray}\label{cf}
	C.f({h})&=&(x-y)f({h})\nonumber\\&=&f(h-2)x-f(h+2)y\nonumber\\&\equiv& \frac{1}{2}(f({h}-2)-f({h}+2)){B}.
	\end{eqnarray}
	
		 \begin{eqnarray}\label{bf}
	B.f({h})&=&(x+y)f({h})\nonumber\\&=&f(h-2)x+f(h+2)y\nonumber\\&\equiv& \frac{1}{2}(f({h}-2)+f({h}+2)){B}.
	\end{eqnarray}
	
	 Note that \begin{align}\label{xb}
	xB&=\frac{1}{2}(B^{2}+CB)\nonumber\\&=\frac{1}{2}(B^{2}+BC+2h)\nonumber\\&\equiv \frac{1}{2}(\mu-{h}^{2}+2{h}),
	\end{align}
	 and 
	 \begin{align}\label{yb}
	 yB&=\frac{1}{2}(B^{2}-CB)\nonumber\\&=\frac{1}{2}(B^{2}-BC-2h)\nonumber\\&\equiv \frac{1}{2}(\mu-{h}^{2}-2{h}).
	\end{align}
Using (\ref{cf}), (\ref{xb}), and (\ref{yb}), we write

\begin{align*}
 C^{2}.f({h})=&\frac{1}{2}(x-y) (f({h}-2)-f({h}+2)){B}\nonumber\\
=&\frac{1}{2}(f(h-4)xB-f(h)xB-f(h)yB+f(h+4)yB)\nonumber\\\equiv& \frac{1}{4}\big((-{h}^{2}+2{h}+\mu)f({h}-4)-2(\mu-{h}^{2})f({h})\\  &+(-{h}^{2}-2{h}+\mu)f({h}+4)\big)
\end{align*}
Hence 
\begin{eqnarray}\label{c2f}
 &C^{2}.f({h})\equiv -\frac{1}{4}\big(({h}^{2}-2{h}-\mu)f({h}-4)\\&+2(\mu-{h}^{2})f({h})+({h}^{2}+2{h}-\mu)f({h}+4)\big).\nonumber
\end{eqnarray} 
In particular, if $ f(h) =h^{n}$, then 
\begin{align*}
 (h^{2}-&2h-\mu)(h-4)^{n}=h^{n+2}-2(2n+1)h^{n+1}\\&+\sum_{k=0}^{n-2}\bigg(-\mu {n\choose k}(-4)^{k}-2{n\choose k+1}(-4)^{k+1}\\&+{n\choose k+2}(-4)^{k+2}\bigg)h^{n-k}\\&+(-\mu n(-4)^{n-1}-2(-4)^{n})h-\mu(-4)^{n}. 
\end{align*}  
and 
 
 \begin{align*}
  (h^{2}+&2h-\mu)(h+4)^{n}=h^{n+2}-2(2n+1)h^{n+1}\\&+\sum_{k=0}^{n-2}\bigg(-\mu {n\choose k}(4)^{k}+2{n\choose k+1}(4)^{k+1}\\&+{n\choose k+2}(4)^{k+2}\bigg)h^{n-k}\\&+(-\mu n(4)^{n-1}+2(4)^{n})h-\mu(4)^{n}.
 \end{align*}
  Hence
  \begin{align}\label{c2hn}  C^{2}.{h}^{n}&\equiv a_{n}h^{n}+a_{n-2}h^{n-2}+\cdots+a_{0}\\ \text{where } \nonumber\\
 a_{n}&=-4n^{2}\nonumber\\ \text{and for  }k&=2,4,\ldots,k-3\ (\text{or }k-2)\nonumber\\
 a_{n-k}&=-\frac{1}{2}\bigg(-\mu {n\choose k}4^{k}+2{n \choose k+1}4^{k+1}+{n\choose k+2}4^{k+2}\bigg). \nonumber\end{align}
All other coefficients equal zero.

 \medskip
 
 One more relation that we will need refers to $ CB.{h}^{n} $.
 
 \smallskip

 Let $ f({h})\in \cc[h]\subset\mc $, then using (\ref{bf}), (\ref{xb}), and (\ref{yb}), we can  write 
\begin{align*}
CB.f({h})=&\frac{1}{2}(x-y) (f({h}-2)+f({h}+2)){B}\nonumber\\&
\equiv\frac{1}{2}(f(h-4)xB+f(h)xB-f(h)yB-f(h+4)yB).
\end{align*}
Hence 
\begin{equation}\label{cbf}
CB.f({h})\equiv \frac{1}{4}\big((-{h}^{2}+2{h}+\mu)f({h}-4)+4{h}f({h})+({h}^{2}+2{h}-\mu)f({h}+4)\big).
\end{equation} 
In particular, if $ f(h) =h^{n}$, then  
\begin{align*}
 (-h^{2}+&2h+\mu)(h-4)^{n}=-h^{n+2}+2(2n+1)h^{n+1}\\&+\sum_{k=0}^{n-2}\bigg(\mu {n\choose k}(-4)^{k}+2{n\choose k+1}(-4)^{k+1}\\&-{n\choose k+2}(-4)^{k+2}\bigg)h^{n-k}\\&+\big(\mu n(-4)^{n-1}+2(-4)^{n}\big)h+\mu(-4)^{n},
\end{align*} 
and
\begin{align*} (h^{2}+&2h-\mu)(h+4)^{n}=h^{n+2}+2(2n+1)h^{n+1}+\\&\sum_{k=0}^{n-2}\bigg(-\mu {n\choose k}(4)^{k}+2{n\choose k+1}(4)^{k+1}+{n\choose k+2}(4)^{k+2}\bigg)h^{n-k}\\&+\big(-\mu n(4)^{n-1}+2(4)^{n}\big)h-\mu(4)^{n}.\end{align*}

 Hence 
 \begin{align}\label{cbhn}
CB.{h}^{n}&\equiv a_{n+1}h^{n+1}+a_{n-1}h^{n-1}+\cdots+a_{1}h\\ \text{where } \nonumber\\
 a_{n+1}&=2(n+1)\nonumber\\ \text{and for  }k&=1,3,\ldots,k-2\ (\text{or }k-1)\nonumber\\
 a_{n-k}&=\frac{1}{2}\bigg(-\mu {n\choose k}4^{k}+2{n \choose k+1}4^{k+1}+{n\choose k+2}4^{k+2}\bigg).\nonumber\end{align}

  \begin{lem}\label{1inN}
 	Let $ \lambda\in \cc \setminus2\zzz $. Consider  a non-zero submodule $ N $ of $ \mc $. If there is a non-zero element $ 0\neq f(h) \in N$, then $ 1\in N $. 
 \end{lem}
\begin{proof}
	Let $ 0\neq v=f({h})\in N $; our claim is to show that $ {1}\in N $.
	
We will prove this fact by induction on the degree of $ f({h}) $. For the base of induction, we will consider $ i=1 $, let $ 0\neq v={h}+a\in N $, where $ a\in \cc $. If $ a\neq 0 $ then
 \begin{align*}
C^2.v\equiv-4h \in N\end{align*}  which implies that $ 0\neq a \in N $. If $ a=0 $, then $ v=h\in N $ and   
\begin{align*}
CB.v&\equiv 4h^2-2\mu\in N, \end{align*} and since $ {h} \in N$ then $ {h}^{2}\in N $, hence $ 0\neq \mu \in N  $, hence $ {1}\in N  $.

 For the  induction step, let \[ v=f({h})={h}^{n}+b_{1}{h}^{n-1}+b_{n-2}{h}^{n-2}+\cdots+b_{0} \in N.\] Using (\ref{c2hn}), for $ n>3$, we have  
 
  \begin{align*}
C^{2}.v\equiv&- \frac{1}{2}\bigg(8n^{2}{h}^{n}+\bigg(-\mu{n \choose 2}4^{2}+2{n \choose 3}4^{3}+{n\choose 4}4^{4}\bigg){h}^{n-2}\\&+\cdots+8b_{n-1}(n-1)^{2}{h}^{n-1}+ \cdots+8b_{n-2}(n-2)^{2}{h}^{n-2}+\cdots\bigg).
\end{align*}
For $ n=2,3 $ we have 
\begin{align*}
C^{2}.(h^{2}+b_{1}h+b_{0})&\equiv-16h^{2}-4b_{1}h+8\mu \\
C^{2}(h^{3}+b_{2}h^{2}+b_{1}h+b_{0})&\equiv-36h^{3}+(24\mu-64-4b_{1})h+8b_{2}\mu 
\end{align*}
Let $v_{1}= f_{1}({h})=C^{2}.v+4n^{2}.v \in N$, then $f_{1}({h}) $ is a polynomial of degree less than $ n $. If $ C^{2}.v\neq-4n^{2}.v $, then we have found a non-zero polynomial of degree less than $ n $ in $ N $, so induction applied to $ f_{1}(h) $. Otherwise, we have  \[-4n^{2}b_{n-1}=-4(n-1)^{2}b_{n-1},\] which implies that $ b_{n-1}=0 $. For $ n>3 $ we have 
\[
-\mu{n \choose 2}4^{2}+2{n \choose 3}4^{3}+{n\choose 4}4^{4}+8b_{n-2}(n-2)^{2}=8b_{n-2}n^{2}
\]
hence
\begin{align*} 
&-\mu{n \choose 2}4^{2}&+2{n \choose 3}4^{3}+{n\choose 4}4^{4}+8b_{n-2}n^{2}-32b_{n-2}n+32b_{n-2}\\
&=8b_{n-2}n^{2},
\end{align*} 
which implies that  
\begin{equation}\label{a2byc2}
b_{n-2}=\frac{-\mu{n \choose 2}+8{n \choose 3}+16{n\choose 4}}{2(n-1)}.
\end{equation}  
For $ n=2 $ we have 
\begin{align*}
b_{0}=-\frac{1}{2}\mu.
\end{align*} 
For $ n=3 $ we have 
\begin{align*}
b_{1}=\frac{8-3\mu}{4}.
\end{align*} 
Now if we have failed to produce a non-zero polynomial of degree less than $ n $ using $ v_{1} $, we can use another action to get such element. Using (\ref{cbhn}), for $ n>2 $, we have 
\begin{align*}
CB.v\equiv&\frac{1}{2}\bigg(4(n+1)h^{n+1}+\bigg(-4n\mu+2{n\choose 2}4^{2}+{n \choose 3}4^{3}\bigg)h^{n-1}\\&+\cdots+4b_{n-2}(n-1)h^{n-1}+\cdots\bigg).
\end{align*}
For $ n=2 $ 
\begin{align*}
CB.(h^{2}+b_{0})\equiv 6h^{3}+(16-4\mu+2b_{0})h.
\end{align*} 

Now consider the element $ v_{2}=f_{2}({h})=CB.v-2(n+1)h.v $, then $ f_{2}({h}) $ is a polynomial of degree less than $ n $. If $ CB.v\neq2(n+1)h.v $, then we have found a non-zero polynomial of degree less than $ n $ in $ N $, so induction applied to $ f_{2}(h) $. Otherwise, for $ n>2  $  \[-4n\mu+2{n\choose 2}4^{2}+{n \choose 3}4^{3}+4b_{n-2}(n-1)=4(n+1)b_{n-2},\] which implies that 
\begin{equation}\label{a2bycb}
b_{n-2}=\frac{-n\mu +8{n \choose 2}+16{n \choose 3}}{2}.
\end{equation} 
For $ n=2 $ we have 
\begin{align}
b_{0}=4-\mu.
\end{align}
 In fact  $ v_{1} $ and $ v_{2} $ cannot be both zero. To see this, assume that $ v_{1} $ and $ v_{2} $ are both zero, then using (\ref{a2byc2}) and (\ref{a2bycb}), for $ n>3 $  we have  
 \[
 \frac{-\mu{n \choose 2}+8{n \choose 3}+16{n\choose 4}}{2(n-1)}=\frac{-n\mu +8{n \choose 2}+16{n \choose 3}}{2}.
 \] 
Multiplying both side by 12, we have  
\begin{align*}
-3n\mu +8n(n-2)+4n(n-2)(n-3)\\=-6n\mu +24n(n-1)+16n(n-1)(n-2)
\end{align*}
and so
\begin{align*} 
&&-3n\mu +4n^3-12n^2+8n&=-6n\mu+16n^3-24n^2+8n\\&\text{hence}
&3n\mu &=12n^3-12n, 	
\end{align*}
since $ n>0 $ 
\[
\mu  =4(n^2-n),
\] 
which implies that 
\begin{align}\label{n value}
\lambda=-2n,2n-2,
\end{align} 
which is not the case.

For $ n=2 $ we have  
\begin{align*}
-\frac{1}{2}\mu=4-\mu \mbox{ hence }
\mu=8.
\end{align*}
Then $\lambda=-4,2$  which is not the case. 

For $ n=3 $ we have  
\begin{align*}
\frac{8-3\mu}{4}=\frac{-3\mu+40}{2} \mbox{ hence }
\mu=24.
\end{align*} 
Then $\lambda=-6,4$, which is again not the case. Thus $ 1\in N $
\end{proof}

%
%

%
%
%
%
%

%
%
%

%
%
%
%
%
%
%
%
 
 \begin{cor}\label{uni of r(h)}
 	Let $ \lambda\in 2\zzz\ $, then there is a uniquely determined monic non-constant polynomial of degree $ n $ (as in (\ref{n value})), say $ r(h)\in \mc $, such that 
 	\begin{align}\label{rhact}
 	C^{2}.r(\h)&\equiv-4n^{2}r(\h)\nonumber\\ CB.r(\h)&\equiv2(n+1)\h r(\h)\nonumber\\BC.r(\h)&\equiv2n\h r(\h)\\B^{2}.r(\h)&\equiv(-\h^{2}-4n)r(\h).\        \quad \quad\Box\nonumber
 	\end{align} 
 \end{cor}
 \begin{cor}\label{pru cor}
 	Let $ \lambda\in 2\zzz $, and $ r(h) $ as in Corollary \ref{uni of r(h)}. Let $ N $ be any non-zero submodule of $ \mc $. If $ f(h)\in N $ with $ f(h)\neq \gamma r(h) $ for all $ \gamma \in \cc $, then 
 	\begin{enumerate}
 		\item If $ \deg(f)\leq n $, then $ N=\mc $.
 		\item If $ \deg(f)> n $, then either $ r(h)\in N $ or $ N=\mc $ .$ \qquad \Box $
 	\end{enumerate} 
 \end{cor}

Note that if $ \lambda\in 2\zzz $, then the polynomial $ r(h) $ (as in (\ref{uni of r(h)})) has degree $ n $, where  \[n= \begin{cases}
\frac{-\lambda}{2} & \text{if }\lambda<0\\
\frac{\lambda+2}{2} & \text{if } \lambda\geq 0 
\end{cases} \]
Moreover, $ \lambda\in 2\zzz $ if and only if $ \mu =4(n^{2}-n) $.

 Next, we can argue in the same way as just above, to evaluate the action on the element $ g(h)B\in \mc $.\\
 Let $ g({h})B\in \cc[h]B\subset\mc $, then

 \begin{eqnarray}\label{cg}
 C.g({h})B&=&(x-y)g({h})B\nonumber\\&=&g(h-2)xB-g(h+2)yB\nonumber\\&\equiv& \frac{1}{2}\big((\mu-h^{2}+2h)g({h}-2)+(-\mu +h^{2}+2h)g({h}+2)\big).
 \end{eqnarray}
 
 Using (\ref{cg}), we can write

 \begin{align}
 C^{2}.g({h})B=&\frac{1}{2}(x-y) \big((\mu-h^{2}+2h)g({h}-2)\\ \nonumber
 +&(-\mu +h^{2}+2h)g({h}+2)\big)\nonumber\\
 =&\frac{1}{2}\big((\mu-8-h^{2}+6h)g(h-4)x+(-\mu+h^{2}-2h)g(h)x\nonumber\\&-(\mu-h^{2}-2h)g(h)y-(-\mu+8+h^{2}+6h)g(h+4)y\big)\nonumber\\\equiv& \frac{1}{4}\big((-{h}^{2}+6{h}+\mu-8)g({h}-4)+2({h}^{2}-\mu)g({h})\\  &-({h}^{2}+6{h}-\mu+8)g({h}+4)\big)B.\nonumber
 \end{align}
 In particular, if $ g(h) =h^{l}$, then 
 \begin{align}\label{c2hnB}  
 C^{2}.{h}^{l}B&\equiv(a_{l}h^{l}+a_{l-2}h^{l-2}+\cdots+a_{0})B\\ \text{where } \nonumber\\
 a_{l}&=-4(l+1)^{2}\nonumber\\ \text{and for  }k&=2,4,\ldots,k-3\ (\text{or }k-2)\nonumber\\
 a_{l-k}&=-\frac{1}{2}\bigg((8-\mu) {l\choose k}4^{k}+6{l \choose k+1}4^{k+1}+{l\choose k+2}4^{k+2}\bigg). \nonumber\end{align}
 
 Another relation that we will need refers to $ CB.({h}^{l}B) $.
 
 Let $ g({h})B\in \cc[h]B\subset\mc $; then
 \begin{align}\label{bg}
 B.g({h})B&\equiv\frac{1}{2}\big((\mu-h^{2}+2h)g({h}-2)-(-\mu+h^{2}+2h)g({h}+2)\big).
 \end{align}
 Using (\ref{xb}), (\ref{yb}), and (\ref{bg}), we can write 
 \begin{align}
 CB.g({h})=&\frac{1}{2}(x-y)\big((\mu-h^{2}+2h)g({h}-2)-(-\mu+h^{2}+2h)g({h}+2)\big)\nonumber\\&
 \equiv\frac{1}{2}\big((\mu-h^{2}+6h-8)g(h-4)-(-\mu+h^{2}-2h)+g(h)x\nonumber\\&-(-\mu-h^{2}-2h)g(h)y-(-\mu+h^{2}+6h+8)g(h+4)y\big).
 \end{align}
 Hence 
 \begin{eqnarray}\label{cbg}
 &CB.g({h})B\equiv \dfrac{1}{4}\big((-{h}^{2}+6{h}+\mu-8)g({h}-4)+4{h}g({h})\\&+({h}^{2}+6{h}-\mu+8)g({h}+4)\big)B.\nonumber
 \end{eqnarray} 
 In particular, if $ g(h) =h^{l}$, then   
 \begin{align}\label{cbhnB}
 CB.({h}^{l}B)&\equiv (a_{l+1}h^{l+1}+a_{l-1}h^{l-1}+\cdots+a_{1}h)B\\ \text{where } \nonumber\\
 a_{l+1}&=2(l+2)\nonumber\\ \text{and for  }k&=1,3,\ldots,k-2\ (\text{or }k-1)\nonumber\\
 a_{l-k}&=\frac{1}{2}\bigg((8-\mu) {l\choose k}4^{k}+6{l \choose k+1}4^{k+1}+{l\choose k+2}4^{k+2}\bigg).\nonumber\end{align}

 \begin{lem}\label{BinN}
 	Let $ \lambda\in \cc \setminus2\zzz $. Suppose $ N $ is a non-zero submodule of $ \mc $. If there is a non-zero element $ 0\neq g(h)B \in N$, then $ B\in N $. 
 \end{lem}
 \begin{proof}
 	Let $ \lambda\in \cc \setminus2\zzz $ and assume that $ N $ is a nonzero submodule of $ \mc $. 
If $  0\neq u=g(h)B\in N $, then we are going to prove $ B\in N $. We will prove this fact by induction on the degree of $ g(h) $. For the base of induction, we will consider $ i=1 $. Let $ 0\neq u=({h}+a)B\in N $, where $ a\in \cc $. If $ a\neq 0 $ then \begin{align*}
 	C^2.u=(-16h-4a)B \in N,
 \end{align*}  
 which implies that $ 0\neq (12a)B \in N $. If $ a=0$, then $ u=hB\in N $ and  
  \begin{align*}
 	CB.u\equiv \big(6h^{2}+4(8-\mu)\big)B, 
 	\end{align*} 
 	and since $ {h} B\in N$ then $ {h}^{2}B\in N $. But $ \lambda  $ is not an even integer, so that $ 0\neq 4(8-\mu)B \in N  $, hence $ {B}\in N  $. 
 	
 	For the induction step, let $  u=g(h)B=\big({h}^{l}+b_{1}{h}^{l-1}+b_{l-2}{h}^{l-2}+\cdots+b_{0}\big)B \in N$. Using (\ref{c2hnB}), for $ l>3 $, we will have   
 	\begin{align*}
 	C^{2}.u\equiv&- \frac{1}{2}\bigg(8(l+1)^{2}{h}^{l}+\bigg((8-\mu){l \choose 2}4^{2}+6{l \choose 3}4^{3}+{l\choose 4}4^{4}\bigg){h}^{l-2}\\&+\cdots+8b_{l-1}(l)^{2}{h}^{l-1}+ \cdots+8b_{l-2}(l-1)^{2}{h}^{l-2}+\cdots\bigg).
 	\end{align*}
 	For $ l=2,3 $, we will have 
 	\begin{align*}
 	&C^{2}.(h^{2}+b_{1}h+b_{0})B\equiv\big(-36h^{2}-16b_{1}h-(8(8-\mu)+4b_{0})\big)B \\
 	&C^{2}(h^{3}+b_{2}h^{2}+b_{1}h+b_{0})B\equiv\big(-64h^{3}-36b_{2}h^{2}\\
 	&-(24(8-\mu)+192+16b_{1})h\\&-(8(8-\mu)b_{2}+4b_{0})\big)B 
 	\end{align*}
 	Let $u_{1}= f_{1}({h})B=C^{2}.u-(-4(l+1)^{2}.u) \in N$, then $f_{1}({h}) $ is a polynomial of degree less than $ l $. If $ C^{2}.u\neq-4(l+1)^{2}.u $, then we have found a non-zero polynomial of degree less than $ l $, so induction applied to $ u_{1} $. \\ Otherwise  \[-4(l+1)^{2}b_{l-1}=-4(l)^{2}b_{l-1},\] which implies that $ b_{l-1}=0 $, and for $ l>3 $ we have \begin{align*}
 	(8-\mu){l \choose 2}4^{2}+6{l \choose 3}4^{3}+{l\choose 4}4^{4}+8b_{l-2}(l-1)^{2}=8b_{l-2}(l+1)^{2},
 	\end{align*} 
 	which implies that  
 	\begin{equation}\label{a2byc2B}
 	b_{l-2}=\frac{(8-\mu){l \choose 2}+24{l \choose 3}+16{l\choose 4}}{2(l)}.
 	\end{equation}  
 	For $ l=2 $ we have
 	 \begin{align*}
 	b_{0}=\frac{8-\mu}{4}.
 	\end{align*} 
 	For $ l=3 $ we have 
 	\begin{align*}
 	b_{1}=\frac{16-\mu}{2}.
 	\end{align*} 
 	
 	Now if we failed to produce a non-zero element with polynomial has a degree less than $ l $ using $ u_{1} $, we can use another action to get such element. Using (\ref{cbhnB}), for $ l>2 $, we have
  \begin{align*}
 	CB.u\equiv&\frac{1}{2}\bigg(4(l+2)h^{l+1}+\bigg(4n(8-\mu)+6{l\choose 2}4^{2}+{l \choose 3}4^{3}\bigg)h^{l-1}\\&+\cdots+4b_{l-2}nh^{l-1}+\cdots\bigg).
 	\end{align*}
If $ l=2 $, we have
 \begin{align*}
 	CB.(h^{2}+b_{0})B\equiv\left( 8h^{3}+(48+4(8-\mu)+4b_{0})h\right) B
 	\end{align*} 
 	Now we can consider an element $ u_{2}=f_{2}({h})B=CB.u-2(l+2)h.u $. We have that $ f_{2}({h}) $ is a polynomial of degree less than $ l $. If $ CB.u\neq2(l+2)h.u $, then we have found a non-zero element in $ N $ with polynomial of degree less than $ l $, so our induction applied to $ u_{2} $. 
 	
 	Otherwise,  
 	\[
 	4n(8-\mu)+6{l\choose 2}4^{2}+{l \choose 3}4^{3}+4b_{l-2}(l)=4(l+2)b_{l-2},
 	\] 
 	which implies that
 	 \begin{equation}\label{a2bycbB}
 	b_{l-2}=\frac{(8-\mu)l +24{l \choose 2}+16{l \choose 3}}{2}.
 	\end{equation} 
 	For $ l=2 $ 
 	\begin{align}
 	b_{0}=20-\mu.
 	\end{align} 
 	In fact  $ u_{1} $ and $ u_{2} $ cannot be both zero. To see this, assume that $ v_{1} $ and $ v_{2} $ are both zero, then using (\ref{a2byc2B}) and (\ref{a2bycbB}), for $ l>3 $ we have  
 	 \[
 	 \frac{(8-\mu){l \choose 2}+24{l \choose 3}+16{l\choose 4}}{2(l)}=\frac{(8-\mu)l +24{l \choose 2}+16{l \choose 3}}{2}.,\]
 	  Multiplying both side by 12, we will have
 	\begin{align*} 
 	16n^{3}+24n^{2}+8n-6n\mu=4n^3-4n-3n\mu+3\mu 	
 	\end{align*}
 	which implies that 
 	\begin{align*}
 	(l+1)\mu=4n(l+1)^{2}.
 	\end{align*}
 	Since $ l>0 $, we have  $\mu  =4(l^2+l),$ which means that $\lambda$ is an even integer, and this is not the case. 
 	
 	For $ l=2 $ we have 
 	\begin{align*}
 	\frac{8-\mu}{4}=20-\mu,\mbox{ hence }\mu=24.
 	\end{align*} 
 	It follows that 
 	$\lambda=4,-6$, which is not the case. 
 	For $ l=3 $ we have
 	 \begin{align*}
 	\frac{16-\mu}{2}=\frac{-3\mu+112}{2},\mbox{ hence }
 	\mu=48,
 	\end{align*} 
 	which means that $\lambda=6,-8$, which is again not the case. Hence if $ 0\neq v=g(h)B \in  N $ then $ B\in N $.
 \end{proof}

 We can see from the previous proof that the equations $ u_{1}=0 $ and $ u_{2}=0 $ has unique solution when $ \lambda\in 2\zzz $. Moreover, if $ r=r(h)+g(h)B $ such that $ C^{2}.r=-4n^{2}r $ and $ CB.r=2(n+1)hr $ is satisfied, then $ g(h) $ has the degree $ l=n-1 $.
 \begin{cor}\label{uni of r(h)B}
 	Let $ \lambda\in 2\zzz $, then there is a uniquely determined monic polynomial $ r^{*}(h)\in \mc $ of degree $ l =n-1$ such that 
 	\begin{align}\label{rhactB}
 	C^{2}.r^{*}(h)B&=-4(l+1)^{2}r^{*}(h)B\nonumber\\ CB.r^{*}(\h)B&=2(l+2)\h r^{*}(h)B\nonumber\\BC.r^{*}(h)B&=2(l+1)\h r^{*}(\h)B\\B^{2}.r^{*}(h)B&=(-\h^{2}-4(l+1))r^{*}(h)B.\        \quad \quad\Box\nonumber
 	\end{align}
 	Note that the element $ C.r(h) $ satisfies the above conditions, where $ l=n-1 $, and since $ r^{*}(h) $ is uniquely determined, we can say that $ r^{*}(h)=C.(r(h)) $.  
 \end{cor}
 \begin{cor}\label{pru corB}
 	Let $ \lambda\in 2\zzz $, $ r(h) $ as in Corollary \ref{uni of r(h)}, and let $ N $ be any non-zero submodule of $ \mc $. If $ g(h)B\in N $ with $ g(h)\neq \gamma C.r(h) $ for all $ \gamma \in \cc $, then 
 	\begin{enumerate}
 		\item If $ \deg(g(h))\leq n-1 $, then $ N=\mc $.
 		\item If $ \deg(g(h))> n-1 $, then either $ C.r(h)\in N $ or $ N=\mc $ $ \Box $
 	\end{enumerate} 
 \end{cor}

 \begin{cor}\label{uni of r(h)+Cr(h)}
	Let $ \lambda\in 2\zzz\ $, then  the elements $ u_{(\alpha_{1},\alpha_{2})}=\alpha_{1}r(h)+\alpha_{2}(C.r(h)) $, where $ \alpha_{1},\alpha_{2}\in \cc $,  are  the only elements in $  \mc $, such that 
	\begin{align}\label{rh+c.rh act}
	C^{2}.u_{(\alpha_{1},\alpha_{2})}&\equiv-4n^{2}u_{(\alpha_{1},\alpha_{2})}\nonumber\\ CB.u_{(\alpha_{1},\alpha_{2})}&\equiv2(n+1)\h u_{(\alpha_{1},\alpha_{2})}\nonumber\\BC.u_{(\alpha_{1},\alpha_{2})}&\equiv2n\h u_{(\alpha_{1},\alpha_{2})}\\B^{2}.u_{(\alpha_{1},\alpha_{2})}&\equiv(-\h^{2}-4n)u_{(\alpha_{1},\alpha_{2})}.
	\end{align} 
	where $ r(h) $ and $ n $ as in (\ref{uni of r(h)}).$ \     \qquad \qquad \qquad \qquad \qquad \qquad\Box\nonumber $
\end{cor}
 
 Now we are ready to prove Theorem \ref{mc simple}.

\begin{proof}[Proof of Theorem \ref{mc simple}] 

	Assume that $ N $ be a nonzero submodule of $ \mc $. Choose a non-zero element $ v\in N $
	
		\textit{Case 1}:
If $ 0\neq v=f({h}) \in  N $ then by Lemma \ref{1inN}, $ 1\in  N $, hence $  N=\mc. $

 \textit{Case 2}: For the general case, let us assume that $ v=f({h})+g({h})B $ with $ g({h})\neq  0$. Using Case 1, we may assume that $ v=\gamma+g({h})B $, where $ \gamma \in \cc$ and $ g({h})\neq  0$.
 
We have 
\[
B.v=\gamma B+q({h})\] where

 \begin{align}\label{B.g(h)B}
q(h)=B.g({h})B\equiv\frac{1}{2}(g({h}-2)(\mu-{h}^2+2{h})+g({h}+2)(\mu-{h}^2-2{h})),
\end{align}
and $ \deg(q({h}))\geq 2 $. If $ \gamma=0 $ then $ 0\neq B.v=q(h) \in N$, which return us to the first case, and if $ 0\neq \gamma $, then

  \begin{align}\label{gam neq 0}
0\neq w({h})=\frac{1}{\gamma}g(h)B.v-v=\frac{1}{\gamma}g({h})q({h})-\gamma \in N,
\end{align}
which also brings us back to the first case. As a result, $\mc$ is a simple module. 
\end{proof}

	\begin{cor}
		Let $ \lambda\in \cc\setminus 2\zzz $, then $ \mc $ is simple $ \zzz_{2}^{2} $-graded $ \s $-module of rank $ 2 $. $ \qquad\Box $
	\end{cor}

 Now assume $ \lambda\in2\zzz$, and consider the subspace

 \begin{align*}
\nc&= \cc[h]r(\h)\oplus\cc[h](C.r(\h))\\&= \cc[h]r(\h)\oplus\cc[h](r(h-2)-r(h+2))B.
\end{align*}
\begin{lem}
	Let $ \lambda\in 2\zzz $, then $\nc $ is a submodule of the module $ \mc $.
\end{lem}
\begin{proof}
	Let $ u=f(h)r(h)+g(h)(C.r(h)) $, it is clear the action of $ h $ leaves $\nc$ invariant. Now
	
 \begin{align*}
	B.u&=Bf(h)r(h)+Bg(h)(C.r(h))\\
	&\equiv(f^{\prime}(h)B+f^{\prime\prime}(h)C)r(h)+(g^{\prime}(h)B+g^{\prime\prime}(h)C)(C.r(h))\\
	&\equiv f^{\prime}(h)B(r(h))+f^{\prime\prime}(h)C(r(h))+g^{\prime}(h)B(C.r(h))\\
	&+g^{\prime\prime}(h)C(C.r(h))\\
	&\equiv f^{\prime}(h)B(r(h))+f^{\prime\prime}(h)C(r(h))+g^{\prime}(h)BC.r(h)+g^{\prime\prime}(h)C^{2}r(h)\\ 
	&\overset{\text{ by }(\ref{rhact})}{\equiv\joinrel\equiv}f^{\prime}(h)\big(B.r(h)\big)+f^{\prime\prime}(h)\big(C.r(h)\big)\\
	&+\big(2(n+1)hg^{\prime}(h)-4n^{2}g^{\prime\prime}(h)\big)r(h),
\end{align*}
	where
\begin{equation}\label{f prime and double prime}
	 \begin{aligned}
	f^{\prime}(h)&=\frac{1}{2}(f(h-2)+f(h+2))\\
	f^{\prime\prime}(h)&=\frac{1}{2}(f(h-2)-f(h+2)),
	\end{aligned}
\end{equation} and similarly for $ g^{\prime}(h) $ and $ g^{\prime\prime}(h) $.

	Note that the second and third terms are in $\nc$, and so it is enough to show that $ B.r(h)\equiv\frac{1}{2}\big(r(h-2)+r(h+2)\big)B $ is in $\nc$.
	
Using (\ref{c2f}), (\ref{rhact}), and $ \mu=4(n^{2}-n) $, we find that 
\begin{align*}
	&C^{2}.r({h})\equiv -\frac{1}{4}(({h}^{2}-2{h}-\mu)r({h}-4)\\
	&+2(\mu-{h}^{2})r({h})+({h}^{2}+2{h}-\mu)r({h}+4))\\&\equiv-4n^{2}r(\h).
	\end{align*} 
 It follows that
\begin{align}\label{rbyc2}
	&({h}^{2}-2{h}-\mu)r({h}-4)+({h}^{2}+2{h}-\mu)f({h}+4)\\&\equiv(8n^{2}+8n+2h^{2})r(h).
	\end{align} 
Similarly, using (\ref{cbf}), (\ref{rhact}), and $ \mu=4(n^{2}-n) $, we get 
\begin{align*}
	&CB.r({h})\equiv \frac{1}{4}((-{h}^{2}+2{h}+\mu)r({h}-4)\\&+4{h}r({h})+({h}^{2}+2{h}-\mu)r({h}+4))\\&\equiv 2(n+1)\h r(\h).
	\end{align*}
	
Thus
	\begin{align}\label{rbycb}
	&({h}^{2}-2{h}-\mu)r({h}-4)-4{h}f({h})-({h}^{2}+2{h}-\mu)r({h}+4))\\ \nonumber
	&=(-8nh-4h)r(h).
	\end{align}
Adding (\ref{rbyc2}) and (\ref{rbycb}), and canceling out $ 2 $, we get 
\begin{align*}
	({h}^{2}-2{h}-4(n^2-n))r({h}-4)=(h^2-(4n+2)h+(4n^2+4n))r(h)
	\end{align*} 
	or 
	\begin{align*}(h-2n)(h+(2n-2))r(h-4)=(h-2n)(h-(2n+2))r(h),
	\end{align*} 
	which implies that
	 \begin{align}\label{rh&rh-4}
	(h+(2n-2))r(h-4)=(h-(2n+2))r(h).
	\end{align}
If we replace $ h $ by $ h+2 $ in (\ref{rh&rh-4}) we get the relation 
\begin{align}\label{rh-2&rh+2} (h+2n)r(h-2)=(h-2n)r(h+2)
	\end{align} 
Using (\ref{rh-2&rh+2}), we obtain 
\begin{align}\label{rh-2}
	&-\frac{1}{4n}(h-2n)\big(r(h-2)-r(h+2)\big)\nonumber\\&=-\frac{1}{4n}((h-2n)r(h-2)-(h-2n)r(h+2))\nonumber\\&\overset{by (\ref{rh-2&rh+2})}{{=\joinrel=\joinrel=}} -\frac{1}{4n}((h-2n)r(h-2)-(h+2n)r(h-2))\nonumber\\&=r(h-2).
\end{align}
 
Similarly, 
\begin{align}\label{rh+2}
	-\frac{1}{4n}(h+2n)&(r(h-2)-r(h+2))\nonumber\\&=-\frac{1}{4n}((h+2n)r(h-2)-(h+2n)r(h+2))\nonumber\\&\overset{by (\ref{rh-2&rh+2})}{{=\joinrel=\joinrel=}} -\frac{1}{4n}((h-2n)r(h+2)-(h+2n)r(h+2))\nonumber\\&=r(h+2). 
\end{align}

As a result,
	\begin{align}\label{B.r(h)}
B.r(h)=&\frac{1}{2}\big(r(h-2)+r(h+2)\big)B\\&\equiv   \frac{1}{2}\bigg(\frac{(-1)}{4n}(h-2n)\big(r(h-2)-r(h+2)\big)B\nonumber\\&+\frac{(-1)}{4n}(h+2n)(r(h-2)-r(h+2))B\bigg)\nonumber\\\overset{\text{by }(\ref{cf})}{\equiv\joinrel\equiv\joinrel\equiv}&\   \frac{1}{2}\bigg(\frac{(-1)}{4n}(h-2n)(2C.r(h))+\frac{(-1)}{4n}(h+2n)(2C.r(h))\bigg)\nonumber\\\equiv&\frac{(-1)}{2n}h(C.r(h)) \in \nc\nonumber.
\end{align}
Hence $ \nc $ is a submodule of $ \mc $.
\end{proof}

\begin{rem}
	
The relation (\ref{B.r(h)}) leads to the following result. If the first two relations in the Corollaries (\ref{uni of r(h)}), (\ref{uni of r(h)B}), and (\ref{uni of r(h)+Cr(h)}) hold then the last two also hold.  
\end{rem}

Now we are ready to start proving Theorem \ref{mc not simple}.
\begin{proof}
	  Let $ W $ be a non-zero submodule of $ \mc $. Choose a non-zero element $ u=f(h)+g(h)B\in W $. We can apply the actions by $ C^{2} $ or $ CB $ and reduce the degree of the polynomial $ f(h) $ until we get either a constant or $ r(h) $. Hence we can reduce the cases of the element $ u $, up to the scalar multiplication, to the following cases: either $ u=f(h) $, $ u=g(h)B $, $ u=1+g(h)B $, or $ u=r(h)+g(h)B $, where $ g(h)\neq0 $ in Cases 2,3, and 4.
	  
 Case 1: If $ 0\neq u=f(h) \in W$, then by Corollary \ref{pru cor}, either $ W=\mc $ or $ W=\nc $.
 
Case 2: Now let $ u=g(h)B $ where $ g(h)\neq0 $, then $ B.u $ is a non-zero element in $ \cc[h] $ in $ W $, see (\ref{B.g(h)B}), which returns us to the Case 1.

Case 3: If $ u=1+g(h)B $ with $ g(h)\neq0 $, then $ g(h)B.u-u $ is a non-zero element in $ \cc[h] $ in $ W $, see (\ref{gam neq 0}), which also means that either $ W=\mc $ or $ W=\nc $.

Case 4: If $ u= r(h)+g(h)B $ and $ g(h)\neq0 $, then 
\[
C^{2}.u-(-4n^{2}u)=g^{*}(h)B\in W,
\]
for some polynomial $ g^{*}(h) $. If $ g^{*}(h)\neq 0 $, then by Case 2, either $ W=\mc $ or $ W=\nc $. Otherwise, $ C^{2}u\equiv-4n^{2}u $, in this case we can try again with the element
\[ 
CB.u-2(n+1)h.u=g^{**}(h)B\in W,
\] 
for some polynomial $ g^{**}(h) $. If $ g^{**}(h)\neq 0 $, then by Case 2, either $ W=\mc $ or $ W=\nc $. Otherwise, $ CB.u\equiv2(n+1)hu $, which means by Corollary (\ref{uni of r(h)+Cr(h)}), that $u=r(h)+\alpha(C(r(h))) $ where $ \alpha\in \cc^{*} $. In this case, $ \alpha C.u=\alpha C(r(h))-4n^{2}\alpha^{2} r(h) \in W $, which implies that $  \alpha C.u+u=(4n^{2}\alpha^{2}+1)r(h)\in W $. If $ (4n^{2}\alpha^{2}+1) \neq 0$, then $ r(h)\in W $, which implies that $ W=\nc $. Hence, the only case remaining is  $ u=r(h)+\alpha(C(r(h))) $ and $ (4n^{2}\alpha^{2}+1) = 0$, in other words, when  $ u=r(h)\pm\frac{i}{2n}(C.r(h)) $.

	\begin{lem}
	Let $ \lambda\in 2\zzz $, and $ r(h) $ as in Corollary \ref{rhact}, then $ \wc=\cc[h](r(h)+\frac{i}{2n}(C.r(h))) $ and $\wcc=\cc[h](r(h)-\frac{i}{2n}(C.r(h)))  $ are submodules of $ \mc $, which are torsion-free modules of rank one.
\end{lem} 
\begin{proof}
	Using Corollary (\ref{uni of r(h)}) and (\ref{B.r(h)}) we have,
	\begin{align*}
	B.\big(r(h)+&\frac{i}{2n}(C.r(h))\big)= B.r(h)+\frac{i}{2n}(BC.r(h))\\\equiv&\frac{(-1)}{2n}hC.r(h)+ihr(h)\\\equiv&ih\big(\frac{(-1)}{2ni}C.r(h)+r(h)\big)\\=&ih\big(r(h)+\frac{i}{2n}C.r(h)\big)
	\end{align*}
	and 
	\begin{align*}
	C.(r(h)+\frac{i}{2n}(C.r(h)))=&C.r(h)+\frac{i}{2n}(C^{2}r(h))\\\equiv&C.r(h)+\frac{i}{2n}(-4n^{2}r(h))\\\equiv&(C.r(h)-i(2nr(h)))\\\equiv&2ni\big(r(h)+\frac{i}{2n}C.r(h)\big)
	\end{align*}
	
	For arbitrary elements, let $ u=f(h)(r(h)+\frac{i}{2n}(C.r(h))) $, it is clear that the action by $ h  $ is invariant. Now \begin{align*}
	B.u=&B\big(f(h)\big(r(h)+\frac{i}{2n}(C.r(h))\big) \big)\\\equiv& \big(f^{\prime}(h)B+f^{\prime\prime}(h)C\big)\big(r(h)+\frac{i}{2n}(C.r(h))\big)\\\equiv&ihf^{\prime}(h)(r(h)+\frac{i}{2n}C.r(h))-2nif^{\prime\prime}(h)(r(h)+\frac{i}{2n}C.r(h))\\\equiv&\big(ihf^{\prime}(h)-2nif^{\prime\prime}(h)\big)\big(r(h)+\frac{i}{2n}C.r(h)\big)
	\end{align*}
	where $ f^{\prime}(h) $ and $ f^{\prime\prime } (h)$ as in \ref{f prime and double prime}. Also 
	\begin{align*}
	C.u=&C\big(f(h)(r(h)+\frac{i}{2n}(C.r(h))) \big)\\\equiv& \big(f^{\prime\prime}(h)B+f^{\prime}(h)C\big)\big(r(h)+\frac{i}{2n}(C.r(h))\big)\\\equiv&ihf^{\prime\prime}(h)(r(h)+\frac{i}{2n}C.r(h))-2nif^{\prime}(h)(r(h)+\frac{i}{2n}C.r(h))\\\equiv&\big(ihf^{\prime\prime}(h)-2nif^{\prime}(h)\big)\big(r(h)+\frac{i}{2n}C.r(h)\big),
	\end{align*} 
	which belongs to $ \wc $. Hence $ \wc $ is a submodule of $ \nc $. A similar calculation shows that $ \wcc $ is also another submodule of $ \nc $.
	
\end{proof}

Hence $ \mc $ has exactly $ 3 $ proper submodules, $ \nc,\wc, \text{ and }\wcc $. Moreover,  $\nc $ is the unique maximal submodule of the module $ \mc $. It is clear that $ \nc=\wc + \wcc $; suppose that $ 0\neq v\in \wc\cap\wcc $. Then $ v  $ can be written as $ v=f(h)(r(h)+\frac{i}{2n}(C.r(h))) $ for some $ 0\neq f(h)\in \cc[h] $, and $ V=g(h)(r(h)-\frac{i}{2n}(C.r(h))) $. It follows that $ f(h)g(h)r(h) \in \wc\cap \wcc$, which means that $ \wc \cap \wcc$ is either $ \mc $ or $ \nc $. This is a contradiction and so $ \nc=\wc \oplus \wcc $.

\end{proof}

\begin{cor}
		Let $ \lambda\in 2\zzz  $, then $ \nc $ is a $ \zzz_{2}^{2} $-graded simple $ \s $-module.
\end{cor}
\begin{proof}
	Since $ \wc $ and $ \wcc $ are of rank $ 1 $, then they are not graded submodules, see \cite{BKS}, which are the only submodules of $ \nc $. But since $ \nc $ is the unique maximal submodule of the graded module $ \mc $, then $ \nc $ is a graded module which has no graded proper submodules. Hence $ \nc  $ is graded-simple.
\end{proof}
 
\begin{rem}
	   Let $ \lambda\in 2\zzz $, and consider the quotient module  $ V=\mc/\nc =\cc[h]\oplus\cc[h]B/\cc[h]r(h)\oplus\cc[h]r^{*}(h)B $, where $ r^{*}(h)B=C.r(h) $. Since the polynomials $ r(h),\ r^{*}(h) $ have degrees $ n,\ n-1 $, respectively, the module $ V $ is a finite-dimensional $ \s $-module, with $ \dim (V) =2n-1$, hence a weight module. Moreover, since the module $ \nc $ is maximal in $ \mc $,  $ V $ is simple. Using \cite[Theorem 3.32]{Mazorchuk}, we have
	    \[
	    V\cong V(2n-2)=\begin{cases}
    V(\lambda)& \text{ If } \lambda\geq 0,\\
    V(-(\lambda+2))& \text{ If } \lambda<0 .
    \end{cases} 
    \]  
\end{rem}
    
    \section{$ \zzz $-gradings of torsion-free $ \s $-modules of finite rank.}
    In this section, we prove the following.
    
    \begin{thm}
    	Any simple torsion-free $ \s $-module of finite rank is not a $ \zzz $-graded $ \s $-module.
    \end{thm}
\begin{proof}
	Let $ M $ be any simple torsion-free $ \s $-module of finite rank $ n $. Assume that $ M $ is a $ \zzz $-graded $ \s $-module, that is, $ M=\underset{i\in \zzz}{\bigoplus}M_{i} $. Since $ h $ has degree $ 0 $ in $ U(\s) $, it follows that $ \cc[h]M_{i}\subset M_{i} $ for all $ i\in \zzz $. Hence $ M_{i} $ is also $ \cc[h] $-submodule of some finite rank $ n_{i}\leq n $, which implies that $ M=\underset{i\in I}{\bigoplus}M_{i} $, where $ I=\left\lbrace i\in\zzz\mid M_{i}\neq \left\lbrace 0\right\rbrace \right\rbrace $, where $ \left| I\right| =r  $ for some positive integer $ r\leq n $.  Let $ t\in I $ be maximal. Then $ x.M_{t}\subset M_{t+1}=0$. For any $ 0\neq v\in M_{t} $, then $ x.v=0 $, that is $ yx.v=0 $. Since $ M $ is simple, then the Casimir element $ c $ act as a scalar on $ M $. Hence $ c.v=((h+1)^{2}+4yx).v=(h+1)^{2}.v=\alpha v $, for some $ \alpha \in \cc $, which implies that $ M $ is a torsion module. A contradiction.  
\end{proof}
    \begin{cor}
Let $\lambda\in \cc\setminus 2\zzz$, then the module $ \mc $ is not a $ \zzz $-graded $ \s $-module.
    \end{cor}

\end{document}